\documentclass[11pt]{amsart}
\usepackage{amsmath,amsthm,amsfonts,amssymb,latexsym}
\usepackage[all,2cell,dvips]{xy}

\usepackage[dvipsnames]{xcolor}

\newtheorem{theorem}{Theorem}[section]
\newtheorem{lemma}[theorem]{Lemma}
\newtheorem{proposition}[theorem]{Proposition}
\newtheorem{corollary}[theorem]{Corollary}

\theoremstyle{definition}
\newtheorem{definition}[theorem]{Definition}

\theoremstyle{remark}

\newcommand{\N}{\mathbb{N}}
\newcommand{\Z}{\mathbb{Z}}

\newcommand{\R}{\mathbb{R}}

\newcommand{\dcl}{\operatorname{dcl}}

\newcommand{\ovs}{\operatorname{ovs}}

\newcommand{\Cal}{\mathcal}

\renewcommand{\dcl}{\operatorname{dcl}}

\def \<{\langle}
\def \>{\rangle}

\def \((  {(\!(}
\def \)) {)\!)}

\numberwithin{equation}{section}

\allowdisplaybreaks[2]

\usepackage{currfile}
\usepackage{fancyhdr}
\begin{document}

\bibliographystyle{plain}
\title{VC-Density in Pairs of Ordered Vector Spaces}
 
\author{Ayhan G\"{u}nayd\i n}

\address{Bo\u{g}azi\c{c}i \"{U}niversitesi, Istanbul, Turkey}

\email{ayhan.gunaydin@bogazici.edu.tr}

\author{Ebru Nay\.Ir}

\address{Bo\u{g}azi\c{c}i \"{U}niversitesi, Istanbul, Turkey}

\email{ebru.nayir@std.bogazici.edu.tr}



\maketitle

\begin{abstract}
We show that the VC-density of any partitioned formula in a pair of ordered vector spaces is bounded above by twice the number of parameter variables.  We also show that this bound is optimal and, as a by-product, we prove that no dense pair of o-minimal structures is dp-minimal.
\end{abstract}

\section{Introduction}

We study pairs of ordered vector spaces over an ordered field, focusing on the VC-densities of formulas. The concept of VC-dimension was defined in the context of statistical learning by Vapnik and Chervonenkis; hence the acronym VC.  A finer invariant is the VC-density of a formula. We define these concepts in the last section, along with some facts about them. 

%

\medskip
Let $V$ be an ordered vector space over an ordered field $K$, and let $W$ be a proper dense subspace of $V$. We consider $V$ and $W$ as structures in the language $L_{\ovs}:=\{+,-,0,1,<\}\cup \{\ell_\lambda : \lambda\in K\}$, where $1$ is interpreted as a positive element of $W$.  Then $W$ is an elementary substructure of $V$. Also, let $L_{\ovs}^d:=L_{\ovs}\cup\{U\}$ be the enrichment of $L_{\ovs}$ by a unary relation symbol $U$. Letting $(V,W)$ denote the $L_{\ovs}^d$-structure where $U$ is interpreted as $W$, our main result is as follows. 

\begin{theorem}\label{main_doag}
The VC-density in $(V,W)$ of a partitioned $L_{\ovs}^d$-formula  $\phi( x; y)$ is at most $2|y|$.
\end{theorem}

We prove that this is the best possible bound; more precisely, we show that there is a partitioned $L_{\ovs}^d$-formula $\phi(x;y)$, where both $x$ and $y$ are single variables such that the VC-density of $\phi$ is $2$. Actually, that formula allows us to show that no dense pair of o-minimal structures is dp-minimal. 

\medskip
The paper  \cite{VC-density-I}  relates VC-density with \emph{uniform definability of types over finite sets} as defined by Guingona in \cite{Guingona}; we restate this relation in the last section as Theorem~\ref{VC1-UDTFS->density}.  In the light of that, we actually prove the following stronger result.

\begin{theorem}\label{main_th}
The theory of dense pairs of ordered vector spaces over an ordered field $K$ has the VC 2 property.
\end{theorem}

The VC 2 property is a stronger version of the uniform definability of types over finite sets; we introduce it in the last section of the paper.

\medskip
The structure $(V,W)$ is an example of a \emph{dense pair of o-minimal structures} as defined in \cite{densepairs}. We recall the definition of this concept in Section~\ref{dense_pairs_section}, and we also mention some facts about dense pairs of o-minimal structures. 
One of these facts provides a characterization of definable sets in such pairs, stated as Theorem~\ref{definable_sets_dense_pairs} below. In the case of dense pairs of ordered vector spaces, we actually prove the following result.

\begin{theorem}\label{QE-T^d_ovs}
The theory of dense pairs of ordered vector spaces over an ordered field $K$ admits quantifier elimination.
\end{theorem}

This result is a part of 5.8 in \cite{DMS_open_core}, but we present a proof for the sake of completeness. This quantifier elimination replaces a much more complicated description of definable sets proved in an earlier version of the paper. The advantage of the earlier description is that it has the potential to be correct in more general dense pairs of o-minimal structures. However, there is an obstruction to generalize Theorem~\ref{main_doag} to such pairs; at the end of the last section, we show that there are formulas with arbitrarily large VC-densities whose parameter variables have fixed length.


\bigskip\noindent
\emph{Notations and Conventions.} The set $\N$ of natural numbers contains $0$ and $\N_+=\N\setminus\{0\}$. We let the letters $m,n,k,l$ vary in $\Z$, and if $m$ is in $\N$ (or $\N_+$), we simply write $m\geq 0$ (or $m>0$).

\medskip\noindent
We let the letters $x,y,z$ denote tuples of variables, possibly with some decorations. In the same way, we do not use separate notations for elements and tuples of elements of sets. We make sure that this does not cause any confusion. One way to do this is to express the lengths of tuples using absolute value notation. 

\medskip\noindent
For a first-order language $L$, an $L$-structure $\Cal M$, and $A\subseteq M$, we say that a subset of $M^n$ is \emph{$L(A)$-definable} if it is defined by a formula of the form $\phi(x,a)$, where $\phi$ is an $L$-formula and $a$ is a tuple of elements of $A$; we refer to such a formula as an \emph{$L(A)$-formula}.  We write \emph{$L$-definable} and \emph{$L$-formula} instead of $L(\emptyset)$-definable and  $L(\emptyset)$-formula.  We also write \emph{definable in $\Cal M$} in the place of $L(M)$-definable.

\section{Quantifier Elimination}\label{dense_pairs_section}
Let $L\supseteq \{+,-,0,1,<\}$ be a language and $T$ a complete o-minimal $L$-theory extending the theory of ordered abelian groups with a distinguished positive element $1$.   Let $L^{d}:=L\cup \{U\}$, where $U$ is  a new unary relation symbol. A \emph{dense pair} of models of $T$ is an $L^d$-structure $(\mathcal{M},\dots,N)$, where $\mathcal{M}\models T$, $N$ is the universe of an $L^d$-structure $\mathcal{N}$, which is a proper elementary substructure of  $\mathcal{M}$, and $N$ is dense in $M$. We denote such a dense pair simply as $(\mathcal{M},\mathcal{N})$.

\medskip\noindent
It is clear that being a dense pair is expressible in $L^d$ and it is shown in \cite{densepairs} that the $L^d$-theory of such pairs is indeed complete; we denote that theory by $T^d$. The following result on definable sets is a by-product of the proof of completeness.

\begin{theorem}\label{definable_sets_dense_pairs}
Let $(\Cal M,\Cal N)$ be a model of $T^d$. Every $L^d$-definable subset of $M^n$ can be defined by a boolean combination of formulas of the form
   \[
   \exists y_1\cdots\exists y_m(U(y_1)\wedge \dots\wedge U(y_m)\wedge \phi(x_1,\dots,x_n,y_1,\dots,y_m)),
   \]
  where $\phi$ is an $L$-formula.
\end{theorem}

\medskip\noindent
We now focus on the setting of dense pairs of ordered vector spaces. Let $K$ be an ordered field, and let $L_{\ovs}$ be the language defined in the Introduction. Also let $T_{\ovs}$ be the $L_{\ovs}$-theory of ordered vector spaces over $K$ with a distinguished positive element $1$. The theory $T_{\ovs}$ is well known to be o-minimal and to admit quantifier elimination.  

\medskip\noindent
We borrow some results from \cite{densepairs} to prove quantifier elimination for the theory $T_{\ovs}^d$. First, recall that the definable closure in an o-minimal structure gives rise to a pregeometry, which we denote by $\dcl$. We call the independence given by $\dcl$, the \emph{$\dcl$-independence}. Let $(\Cal M,\Cal N)$ be a substructure of a model $(\Cal M',\Cal N')$ of $T^d$. In \cite{densepairs}, the structures $\Cal M$ and $\Cal N'$ are defined to be \emph{free over} $\Cal N$ (\emph{in $\Cal M'$}) if any $A\subseteq M$ that is $\dcl$-independent (in $\Cal M'$) over $\Cal N$ remains $\dcl$-independent over $\Cal N'$. The next result is Corollary 2.7 of \cite{densepairs}.

\begin{theorem}\label{dense_pairs_corollary_2_7}
Let $(\Cal M,\Cal N),(\Cal M',\Cal N')\models T^d$ such that $(\Cal M,\Cal N)\subseteq (\Cal M',\Cal N')$ and $\Cal M$ and $\Cal N'$ are free over $\Cal N$. Then $(\Cal M,\Cal N)\preceq  (\Cal M',\Cal N')$
\end{theorem}




\medskip\noindent
We prove the following for models of $T_{\ovs}^d$.

\begin{proposition}\label{freeness_T_ovs^d}
Let $(V,W)$ be a substructure of a model $(V^{'},W^{'})$ of $T_{\ovs}^d$. Then $V$ and $W^{'}$ are free over $W$ in $V^{'}$.
\end{proposition}

\begin{proof}
Note that $V\cap W^{'}=W$ and the case where $W^{'}= W$ is trivial. We may assume then $W^{'}\setminus W\neq \emptyset$.

\medskip\noindent
Let $v_1,\ldots,v_{n}\in V$ be dcl-independent over $W$, and suppose they are not dcl-independent over $W^{'}$. Without loss of generality, there exist $\lambda_1,\ldots,\lambda_{n-1}\in K$ and $w_1\in W^{'}$ such that
$v_n= \sum_{i=1}^{n-1} \lambda_iv_i +w_1.$
Since $v_n\in V$ is dcl-independent over $W$, we have
 \[ v_n\notin \bigoplus_{i=1}^{n-1} Kv_i +W.\]
Combining the last two observations yields $w_1\in W^{'}\setminus W$. We have on the other hand that $w_1\in V$, contradicting $w_1\notin V\cap W^{'}=W$.
\end{proof}

\medskip\noindent
A direct consequence of Theorem~\ref{dense_pairs_corollary_2_7} and Proposition~\ref{freeness_T_ovs^d} is as follows.

\begin{corollary}\label{modelcompleteness_T^d_ovs}
    The theory $T_{\ovs}^d$ of ordered vector spaces over $K$ is model complete.
\end{corollary}

\medskip\noindent
The following is a special case of Corollary 2.8 of \cite{densepairs}.

\begin{corollary}\label{precompleteness_dense_pairs}
    Let $(V,W)$ be a common substructure of models $(V^{'},W^{'})$ and
$(V^{''},W^{''})$ of $T^d_{\ovs}$. 
Then $(V^{'},W^{'})$ and $(V^{''},W^{''})$ are elementarily equivalent over $V$, that is, they satisfy the same $L^{d}_{\ovs}(V)$-sentences.
\end{corollary} 


\medskip\noindent
We are now ready to prove the quantifier elimination result. 

\begin{proof}[Proof of Theorem~\ref{QE-T^d_ovs}]
    Let $(V^{'},W^{'})$ and $(V^{''},W^{''})$ be two models of $T_{\ovs}^d$ with a common substructure $(V,W)$. Let also $\varphi(x,y)$ be a quantifier free $L_{\ovs}^d$-formula where $x$ denotes a single variable and $y$ denotes any tuple of free variables. If $(V^{'},W^{'})\models \varphi(a,v)$ for $v\in V^{|y|}$ and $a\in V^{'}$, then by Corollary~\ref{precompleteness_dense_pairs} we have $(V^{''},W^{''})\models \varphi(b,v)$ for some $b\in V^{''}$.
Hence $T_{\ovs}^d$ has quantifier elimination.
\end{proof}

\section{VC-density and UDTFS }
Below, $\Cal M$ is a structure in a language $L$ and its theory is denoted by $T$. We deal with \emph{partitioned $L$-formulas} $\phi( x; y)$, which means that the set of free variables of $\phi$ are partitioned into two: \emph{object variables $ x=(x_1,\dots,x_m)$} and \emph{parameter variables $ y=(y_1,\dots,y_n)$}.

\medskip\noindent
For a subset $A\subseteq M^m$, we let $A\cap\phi(\Cal M;y)$ denote the collection 
 \[
  \big\{A\cap\phi(\Cal M;b):b\in M^n\big\}.
 \]
If, for instance, $A$ is finite with $k$ elements, then $A\cap\phi(\Cal M;y)$ has at most $2^k$ members. 
 
\medskip\noindent
The \emph{shatter function} $\pi_\phi:\N\to\N$ of $\phi(x;y)$ is defined as follows: 
 \[
   \pi_\phi(k):=\max\big\{|A\cap\phi(\Cal M;y)|:A\subseteq M^m, |A|=k\big\}.
 \]
If there is $k\geq 0$ with $\pi_\phi(k)<2^k$, then the \emph{VC-dimension} of $\phi(x;y)$ is the largest $d$ such that $\pi_\phi(d)=2^d$. Otherwise we say that the VC-dimension of $\phi(x;y)$ is infinity. 

\medskip\noindent
A result proved independently by  Sauer (in \cite{Sauer}) and Shelah (in \cite{Shelah-VC-duality}) states that if the VC-dimension of $\phi(x;y)$ is finite, say $d$, then the shatter function $\pi_\phi$ is bounded by a polynomial of $k$ of degree $d$. This allows us to define the \emph{VC-density} of $\phi(x;y)$ to be the infimum of the real numbers $r\in\R$ such that $\pi_\phi(k)=O(k^r)$.

\medskip\noindent
We do not actually calculate the VC-density of formulas in the dense pairs of ordered vector spaces. We use its relation with the concept of the definability of types as defined below.

\medskip\noindent
Let $\Delta( x; y)$ be a finite set of partitioned $L$-formulas; let $m=| x|$ and $n=| y|$. Given a subset $B\subseteq M^n$, a set of $L$-formulas that are  either of the form $\phi( x; b)$ or of the form $\neg\phi( x; b)$, where $\phi\in\Delta$ and $ b\in B$ is called a \emph{$\Delta$-type over $B$} if it is maximally consistent in $T$. (In many sources this is called a \emph{complete} $\Delta$-type.) We denote the set of $\Delta$-types over $B$ as $S^\Delta(B)$.

\begin{definition}
A finite set $\Delta( x; y)$ has \emph{uniform definability of types over finite sets (UDTFS) with $d$ parameters (in $\Cal M$)} if there are collections 
 \[
  (\phi^1( y; y_1,\dots, y_d))_{\phi\in\Delta},\dots,(\phi^r( y; y_1,\dots, y_d))_{\phi\in\Delta}
 \] 
 of $L$-formulas such that for every finite $B\subseteq M^n$ and $p\in S^\Delta(B)$ there are $i\in\{1,\dots,r\}$ and $ b_1,\dots, b_d\in B$ such that for every $ b\in B$
 \[
 \phi( x;b)\in p\Longleftrightarrow \Cal M\models\phi^i( b; b_1,\dots, b_d).
 \] 
\end{definition}

\medskip\noindent
The collections $ (\phi^1( y; y_1,\dots, y_d))_{\phi\in\Delta},\dots,(\phi^r( y; y_1,\dots, y_d))_{\phi\in\Delta}$ as in this definition are said to \emph{define $\Delta$}.

\medskip\noindent
It is shown in \cite{Chernikov-Simon-2} that in any structure with NIP, every singleton $\{\phi( x; y)\}$ has UDTFS with $d$ parameters for some $d$. Using Corollary 3.2 of  \cite{dependent}, the same holds for any dense pair of models of an o-minimal theory extending the theory of divisible ordered abelian groups as in Section 3. We are aiming to show that in a pair $(V,W)$ of models of $T_{\ovs}$, any finite $\Delta(x; y)$ with a single object variable actually has UDTFS with $2$ parameters. Quantifier elimination of $T_{\ovs}^d$ enables us to use the relation of UDTFS with another notion, the property of
\textit{breadth} introduced in \cite{VC-density-I} as follows.

\medskip\noindent
A collection $\mathcal{S}$ of sets is said to have \textit{breadth $d$} if for any $A_i\in \mathcal{S}$ with $\bigcap_{i\in I} A_i\neq \emptyset$ and $d\leq |I|<\infty$, there is a subset $J\subseteq I$ with $|J|=d$ and $\bigcap_{i\in I} A_i=\bigcap_{j\in J} A_j$. The following result shows how small breadth guarantees UDTFS with finitely many parameters. 

\begin{theorem}[Lemma 5.2 of \cite{VC-density-I}]\label{VC1-breadth}
Let $\Delta(x;y)$ be a finite set of partitioned formulas and $\mathcal{S}_\Delta$ be the collection of subsets of $M^{|x|}$ defined by the formulas from $\Delta(x;y)$, where the parameters range over $M^{|y|}$. If $\mathcal{S}_\Delta$ has breadth $d$, then $\Delta(x;y)$ has UDTFS with $d$ parameters. 
\end{theorem}

\medskip\noindent
We continue with the following, which is Corollary 5.6 in \cite{VC-density-I}.

\begin{theorem}\label{VC1-Boolean}
Let $\Phi$ be a set of partitioned formulas in a single object variable $x$ such that 
 \begin{itemize}
  \item every partitioned formula with object variable $x$ is equivalent in $T$ to a boolean combination of formulas from $\Phi$. 
  \item every finite subset of  $\Phi$ has UDTFS with $d$ parameters.
 \end{itemize}
Then every finite set of $L$-formulas with a single object variable has UDTFS with $d$ parameters.
\end{theorem}

\medskip\noindent
In \cite{VC-density-I}, a structure satisfying the conclusion of this theorem is said to have the \emph{VC d property}; we will be using this terminology as well. Another result from \cite{VC-density-I} is that every weakly o-minimal structure has the VC 1 property; Theorem 6.1. In particular, any model of $T_{\ovs}$ has the VC 1 property. 

\medskip\noindent
Next we state the relation of the definability of types and VC-density as given in \cite{VC-density-I}.

\begin{theorem}[Corollary 5.8 of \cite{VC-density-I}]\label{VC1-UDTFS->density}
Suppose that $\Cal M$ has the VC d property. Then the VC-density of every formula $\phi( x;y)$ is bounded by $d| y|$.
\end{theorem}

\noindent
In this theorem, $ x$ is of arbitrary length. 

\medskip\noindent
Now we are in a position to prove Theorem~\ref{main_th}, and  
Theorem~\ref{main_doag} is a consequence of it combined with Theorem~\ref{VC1-UDTFS->density} above. 

\begin{theorem}
Any model $(V,W)$ of $T_{\ovs}^d$ has the VC 2 property.
\end{theorem}

\begin{proof}
Fix a single variable $x$, and let $y$ vary among tuples of variables of various lengths. For any $\lambda_0\in K$ and $\lambda\in K^{|y|}$, let $\Phi$ be the collection of the following partitioned $L_{\ovs}^d$-formulas in variables $x$ and $y$:
 \begin{enumerate}
     \item $  x=\lambda_0+\lambda y$,
     \item $x< \lambda_0+\lambda y$,
     \item $U(x+\lambda_0+\lambda y)$,
\end{enumerate}
where $\lambda_0$ denotes $l_{\lambda_0}(1)$, and $\lambda y$ denotes $l_{\lambda_1}(y_1)+\dots+l_{\lambda_n}(y_n)$.
 
\medskip\noindent
Since $T_{\ovs}^d$ eliminates quantifiers by Theorem~\ref{QE-T^d_ovs}, every partitioned $L_{\ovs}^d$-formula with object variable $x$ is equivalent in $(V,W)$ to a boolean combination of formulas from $\Phi$. Therefore by Theorem~\ref{VC1-Boolean}, it remains to show that every finite subset of $\Phi$ has UDTFS with $2$ parameters. So let $\Delta(x;y)$ be a finite subset of $\Phi$. Possibly after adding some dummy variables into some formulas in $\Delta(x;y)$, we may assume that the parameter variables of elements of $\Delta(x;y)$ is a fixed tuple $y$ of length $n$.

\medskip\noindent
Consider $\mathcal{S}_\Delta$, the collection of subsets of $V$ defined by the formulas in $\Delta(x;y)$ with parameters from $V^n$. It is enough to show that $\mathcal{S}_\Delta$ has breadth $2$ by Theorem~\ref{VC1-breadth}. Note that by definition of $\Phi$, each set in $\mathcal{S}_\Delta$ is either a point in $V$, or an initial segment of $V$, or a coset of $W$. 

\medskip\noindent
Let $A_1,\ldots,A_m\in \mathcal{S}_\Delta$ with $m\geq 2$, and suppose that $\bigcap_{j=1}^{m}A_j\neq \emptyset$. Since the intersection of two cosets of $W$ is nonempty if only if they are the same coset of $W$, if there are $A_j$ defined by formulas of the form $(3)$, their intersection is given by only one of them, say $A_{j_3}$.
By the same argument, we may take one of the $A_j$ defined by $(1)$, say $A_{j_1}$, to represent their intersection if such sets exist. If there are $A_j$ defined by formulas of the form $(2)$, then their intersection is the one whose supremum is the minimum among them; denote this initial segment by $A_{j_2}$. Therefore, either $\bigcap_{j=1}^{m}A_j=A_{j_k}$ for some $k=1,2,3$ or $\bigcap_{j=1}^{m}A_j=A_{j_2}\cap A_{j_3}$.
\end{proof}

\subsection*{Optimality of Theorem~\ref{main_doag}} We prove that our result is optimal by showing that the formula $y\leq x<y+1\vee U(x-y)$ has VC-density $2$.  Note that this is a formula in the language of any dense pair of o-minimal structures; we actually prove the following general result.

\begin{proposition}
Let $(\Cal M,\Cal N)$ be a dense pair of o-minimal structures, and let $\phi(x;y)$ be the formula $y\leq x<y+1\vee U(x-y)$. Then for any $k$, calculated in $(\Cal M,\Cal N)$, the quantity $\pi_\phi(k)$ is at least $\frac{k^2+k+2}{2}$.
\end{proposition}

\begin{proof}
Let $A=\{a_1,\dots,a_k\}$ where $a_i+1<a_{i+1}$ for each $i$ and $a_i-a_j\notin N$ for $i\neq j$. Clearly, $\phi(\Cal M,a_i)\cap A=\{a_i\}$ for each $i$.  Using the density of $N$, for each $i< j$, we may find $b$ such that $a_i\in[b,b+1)$ and $a_j\in b+N$. Therefore we have $\phi(\Cal M,b)\cap A=\{a_i,a_j\}$.  Hence the collection $A\cap\phi(\Cal M;y)$ contains all the singletons and two-element sets that are contained in $A$.  Counting also the empty set, we get the desired lower bound.
\end{proof}

\noindent
Putting this together with Theorem~\ref{main_doag}, we get the following.

\begin{corollary}
Let $(V,W)$ be a dense pair of ordered vector spaces. Then the VC-density of $\phi$ is $2$.
\end{corollary}

\noindent
For $n>0$, let $\phi_n(x;y_1,\dots,y_n)$ be the formula $\bigvee_{i=1}^ny_i\leq x<y_i+1 \vee U(x-y_i)$. One may modify the proof of the proposition above to show that $\pi_{\phi_n}(k)=O(k^{2n})$, hence we get that the VC-density of $\phi_n$ in any dense pair of ordered vector spaces is $2n$. Therefore, Theorem~\ref{main_doag} is sharp for every possible tuple $y$ of variables.

\medskip\noindent
It follows from Corollary 3.7 of \cite{Simon_dp_minimal} that no dense pair of o-minimal structures is dp-minimal. Here we present an alternative proof of that, using the idea of the proof of the proposition above. We use the definition of dp-minimality from \cite{VC-density-I}.

\begin{definition}
 A theory $T$ is \emph{dp-minimal} if for any model $\Cal M$ of $T$ and formulas $\phi(x;y)$ and $\psi(x;y)$ with $|x|=1$, there are no sequences $(a_i)_{i\in\N}$ and $(b_i)_{i\in\N}$ from $M^{|y|}$ such that for any $i,j\in\N$ the following set of formulas has a realization in $\Cal M$:
  \[
    \{\phi(x;a_i),\psi(x;b_j)\}\cup\{\lnot \phi(x;a_k):k\neq i\}\cup\{\lnot \psi(x;b_l):l\neq j\}.
  \]
\end{definition}

\begin{proposition}
 Let $T$ be an o-minimal theory. Then $T^d$ is not dp-minimal.
\end{proposition}

\begin{proof}
 Let $(\Cal M,\Cal N)$ be a model of $T^d$. Let $\phi(x;y)$ and $\psi(x;y)$ be the formulas $y<x<y+1$ and $U(x-y)$. Take $(a_i)_{i\in\N}$ from $M$ such that $a_i+1<a_{i+1}$ and $a_i-a_j\notin N$ for $i\neq j$. Also let $b_i=a_i$ for all $i\in\N$. Then given $i,j\in\N$, using the density of $N$ in $M$, we get $c\in(a_i,a_i+1)\cap a_j+N$. It is clear that this $c$ is not in $\bigcup_{k\neq i,l\neq j}(a_k,a_k+1)\cup a_l+N$.
\end{proof}

\medskip\noindent
A related concept is the \emph{dp-rank} of a theory. We do not define this concept, but we use its relation with VC-density: If the dp-rank is at least $n$, then there is a formula $\phi(x,y)$ with $|y|=1$ that has VC-density at least $n$. The next result is Lemma 2.10 of \cite{VC-density-II}.

\begin{lemma}
Let $\Cal M$ be an $L$-structure. Suppose that there are a definable set $X\subseteq M$ and a definable injection $X^n\to M$. Then the dp-rank of $\operatorname{Th}(\Cal M)$ is at least $n$.
\end{lemma}
 
\medskip\noindent
Let $(K,L)$ be a dense pair of real closed fields. Then the transcendence degree of $K$ over $L$ is infinite. Hence for any $n$, we can find an injection $L^n\to K$ that is definable in the pair. According to the lemma above, we get that the dp-rank of the theory of dense pairs of real closed fields is at least $n$, for each $n$. Therefore for each $n$, there is a formula $\phi(x,y)$ with $|y|=1$ that has VC-density at least $n$. Hence it is impossible to generalize Theorem~\ref{main_doag} to dense pairs of o-minimal structures that expand fields.

\bibliography{ref.bib}

\end{document}